\documentclass[12pt]{article}
\usepackage{sigsam, amsmath}

\issue{}
\articlehead{}
\titlehead{Ultrametric Smale's $\alpha$-theory}
\authorhead{Jazz G. Suchen and Josué Tonelli-Cueto}
\usepackage[inline]{enumitem}

\usepackage{amsthm}
\theoremstyle{plain}
\newtheorem{lem}{Lemma}
\newtheorem{prop}[lem]{Proposition}
\newtheorem{theo}[lem]{Theorem}

\theoremstyle{definition}
\newtheorem{defi}{Definition}
\theoremstyle{remark}

\def\bfd{\boldsymbol{d}}

\def\Pd{\mcP_{n,\bfd}}

\DeclareMathOperator{\dist}{dist}

\def\diff{\mathrm{D}}
\def\newton{\mathrm{N}}

\def\Zp{\bbZ_p}

\def\mcP{\mathcal{P}}

\def\bbF{\mathbb{F}}

\def\bbZ{\mathbb{Z}}

\def\bbI{\mathbb{I}}

\begin{document}

\title{Ultrametric Smale's $\alpha$-theory}

\author{Jazz G. Suchen\\
Institute for Interspecies Studies\\
Berlin, Germany\\
\url{jazz.g.suchen@outlook.com}
\and
Josué Tonelli-Cueto\\
Inria Paris \& IMJ-PRG\\
Paris, France\\
\url{josue.tonelli.cueto@bizkaia.eu}
}

\date{}

\maketitle

\begin{abstract}
We present a version of Smale's $\alpha$-theory for ultrametric fields, such as the $p$-adics and their extensions, which gives us a multivariate version of Hensel's lemma.
\end{abstract}

Hensel's lemma~\cite[\S3.4]{gouveabook} gives us sufficient condition for lifting roots mod $p^k$ to roots in $\Zp$. Alternatevely, Hensel's lemma gives us sufficient conditions for Newton's method convergence towards an approximate root. Unfortunately, in the multivariate setting, versions of Hensel's lemma are scarce~\cite{conradhensel}. However, in the real/complex world, Smale's $\alpha$-theory~\cite{dedieubook} gives us a clean sufficient criterion for deciding if Newton's method will converge quadratically. In the $p$-adic setting, Breiding~\cite{breiding2013} proved a version of the $\gamma$-theorem, but he didn't provide a full $\alpha$-theory. In this short communication, we provide an ultrametric version of Smale's $\alpha$-theory for square systems---initially presented as an appendix in~\cite{TC-Strassman}---, together with an easy proof.

In what follows, and for simplicity\footnote{We focus on characteristic zero and polynomials to avoid technical details related to Taylor series.}, $\bbF$ is a non-archimedian complete field of characteristic zero with (ultrametric) absolute value $|~|$ and $\Pd[n]$ the set of polynomial maps
$f:\bbF^n\rightarrow \bbF^n$
where $f_i$ is of degree $d_i$. In this setting, we will consider on $\bbF^n$ the ultranorm given by
$
\|x\|:=\max\{x_1,\ldots,x_n\},
$
its associated distance
$
\dist(x,y):=\|x-y\|,
$
and on $k$-multilinear maps $A:(\bbF^n)^k\rightarrow \bbF^q$ the induced ultranorm, which is given by 
\begin{equation}
    \|A\|:=\sup_{v_1,\ldots,v_k\neq 0}\frac{\|A(v_1,\ldots,v_k)\|}{\|v_1\|\cdots\|v_k\|}.
\end{equation}
In this context, we can define Smale's parameters as follows. Below $\diff_xf$ denotes the differential map of $f$ at $x$ and $\diff_x^kf$ the $k$-linear map induced by the $k$th order partial derivatives of $f$ at $x$.

\begin{defi}[Smale's parameters]
Let $f\in\Pd[n]$ and $x\in\bbF^n$. We define the following:
\begin{enumerate}
    \item[(a)] \emph{Smale's $\alpha$}: $\alpha(f,x):=\beta(f,x)\gamma(f,x)$, if $\diff_x f$ is non-singular, and $\alpha(f,x):=\infty$, otherwise.
    \item[(b)] \emph{Smale's $\beta$}: $\beta(f,x):=\|\diff_xf^{-1}f(x)\|$, if $\diff_x f$ is non-singular, and $\alpha(f,x):=\infty$, otherwise.
    \item[(c)] \emph{Smale's $\gamma$}: $\gamma(f,x):=\sup_{k\geq 2}\left\|\diff_xf^{-1}\frac{\diff_x^kf}{k!}\right\|^{\frac{1}{k-1}}$, if $\diff_x f$ is invertible, and $\gamma(f,x):=\infty$, otherwise.
\end{enumerate}
\end{defi}

If $\diff_xf$ is invertible, then the \emph{Newton operator},
\[\newton_f:x\mapsto x-\diff_xf^{-1}f(x),\]
is well-defined at $x$. For a point $x$, the \emph{Newton sequence} is the sequence $\{\newton_f^k(x)\}$. Note that this sequence is well-defined (i.e., $\newton_f^k(x)$ makes sense for all $k$) if and only if $\diff_{\newton_x^k(x)}f$ is invertible at every $k$. Also note that
\[
\beta(f,x)=\|x-\newton_f(x)\|,
\]
so $\beta$ measures the length of a Newton step.

\begin{theo}[Ultrametric $\alpha$/$\gamma$-theorem]\label{thm:Smalealphatheory}
Let $f\in\Pd[n]$ and $x\in\bbF^n$. Then the following are equivalent:
\[
\text{(}\alpha\text{) }\alpha(f,x)<1~\text{ and }~\text{(}\gamma\text{) } \dist(x,f^{-1}(0))<1/\gamma(f,x)
\]
Moreover, if any of the above equivalent conditions holds, then the Newton sequence, $\{\newton_f^k(x)\}$, is well-defined and it converges quadratically to a non-singular zero $\zeta$ of $f$. More specifically, for all $k$, the following holds:
\begin{center}
\emph{(a)} $\alpha(f,\newton_f^k(x))\leq \alpha(f,x)^{2^{k}}$.\hfill
\emph{(b)} $\beta(f,\newton_f^k(x))\leq \beta(f,x)\alpha(f,x)^{2^{k}-1}$.\hfill
\emph{(c)} $\gamma(f,\newton_f^k(x))\leq \gamma(f,x)$.

    \begin{enumerate*}
    \item[(Q)] $\|\newton_f^k(x)-\zeta\|=\beta(f,\newton_f^k(x))\leq \alpha(f,x)^{2^{k}-1}\beta(f,x)<\alpha(f,x)^{2^{k}-1}/\gamma(f,x)$.
\end{enumerate*} 
\end{center}
\end{theo}

In the univariate $p$-adic setting, we have that for $f\in\Zp[X]$ and $x\in\Zp$, 
\[
\gamma(f,x)\leq 1/|f'(x)|,
\]
since $|1/k!f^{(k)}(x)|\leq 1$ and $|f'(x)|\leq 1$. Therefore we can see that the condition $|f(x)|<|f'(x)|^2$ of Hensel's lemma implies $\alpha(f,x)<1$ for a $p$-adic integer polynomial. In this way, we can see that Theorem~\ref{thm:Smalealphatheory} generalizes Hensel's lemma to the multivariate case.

Moreover, in the univariate setting, we can show the following proposition which gives a precise characterization of Smale's $\gamma$ in the ultrametric setting as the separation between `complex' roots---not only a bound as it happens in the complex/real setting.

\begin{prop}[Ultrametric separation theorem for $\gamma$]\cite[Theorem~3.15]{TC-Strassman}\label{prop:separation}
Fix an algebraic closure $\overline{\bbF}$ of $\bbF$ with the corresponding extension of the ultranorm. Let $f\in \bbF[X]$ and $\zeta\in \overline{\bbF}$ a simple root, then
\begin{equation*}\tag*{\qed}
   \frac{1}{\gamma(f,\zeta)}=\dist(\zeta,f^{-1}(0)\setminus \{\zeta\}).
\end{equation*}
\end{prop}

\section*{Proof of Theorem~\ref{thm:Smalealphatheory}}

The proof of the theorem relies in the following three lemmas, stated for $f\in\Pd[n]$ and $x,y\in\bbF^n$. 

\begin{lem}\label{lem:inversegamma}
%Let $f\in\Pd[n]$ and $x,y\in\bbF^n$. 
If $\gamma(f,x)\|x-y\|<1$, then $\diff_y f$ is invertible and
$\|\diff_y f^{-1}\diff_xf\|=1$.
\end{lem}
\begin{lem}[Variations of Smale's parameters]\label{lem:variationSmaleparameters}
%Let $f\in\Pd[n]$ and $x,y\in\bbF^n$. 
If $\rho:=\gamma(f,x)\|x-y\|<1$, then:
\begin{center}
\emph{(a)} $\alpha(f,y)\leq \max\{\alpha(f,x),\rho\}$. 
\emph{(b)} $\beta(f,y)\leq \max\{\beta(f,x),\|y-x\|\}$. 
\emph{(c)} $\gamma(f,y)=\gamma(f,x)$.
\end{center}
Moreover, if $\|y-x\|<\beta(f,x)$, all are equalities.
\end{lem}
\begin{lem}[Variations along Newton step]\label{lem:variationNewton}
%Let $f\in\Pd[n]$ and $x\in\bbF^n$. 
If $\alpha(f,x)<1$, then:
\begin{center}
\emph{(a)} $\alpha(f,\newton_f(x))\leq \alpha(f,x)^2$.\hfill
\emph{(b)} $\beta(f,\newton_f(x))\leq \alpha(f,x)\beta(f,x)$.\hfill
\emph{(c)} $\gamma(f,\newton_f(x))=\gamma(f,x)$.
\end{center}
In particular, $\newton_f(\newton_f(x))$ is well-defined.
\end{lem}

\begin{proof}[Proof of Theorem~\ref{thm:Smalealphatheory}]
If $\alpha(f,x)<1$, then, using induction and Lemma~\ref{lem:variationNewton}, we obtain that (a), (b) and (c) hold. But then the sequence $\{\newton_f^k(x)\}$ converges since
$\lim_{k\to\infty}\|\newton_f^{k+1}(x)-\newton_f^k(x)\|=0$
and so it is a Cauchy sequence. Finally, (Q) follows from noting that for $l\geq k$
\[
\|\newton_f^l(x)-\newton_f^k(x)\|\leq \alpha(f,x)^{2^{l-k}}\beta(f,\newton_f^k(x))
\]
and taking infinite sum together with the equality case of the ultrametric inequality. In particular, we have $\dist(x,f^{-1}(0))=\|x-\zeta\|=\beta(f,x)<1/\gamma(f,x)$. This shows that ($\alpha$) implies ($\gamma$).

For the other direction, assume that $\dist(x,f^{-1}(0))<1/\gamma(f,x)$. Then $\gamma(f,x)$ is finite, since otherwise $\dist(x,f^{-1}(0))<0$, which is impossible. Let $\zeta\in\bbF^n$ be a zero of $f$ such that $\dist(x,\zeta)<1/\gamma(f,x)$. Then
$
0=f(\zeta)=f(x)+\diff_xf(\zeta-x)+\sum_{k=2}^{\infty}\frac{\diff_x^kf}{k!}(\zeta-x,\ldots,\zeta-x),
$
and so
\[
-\diff_xf^{-1}f(x)=\zeta-x+\sum_{k=2}^{\infty}\diff_xf^{-1}\frac{\diff_x^kf}{k!}(\zeta-x,\ldots,\zeta-x).
\]
Now, the higher order terms satisfy that
$
\left\|\diff_xf^{-1}\frac{\diff_x^kf}{k!}(\zeta-x,\ldots,\zeta-x)\right\|\leq \left(\gamma(f,x)\|\zeta-x\|\right)^{k-1}\|\zeta-x\|<\|\zeta-z\|
$
and so, by the equality case of the ultrametric inequality,
$
\beta(f,x)=\|\zeta-x\|<1/\gamma(f,x),
$
as desired.
\end{proof}

Now, we prove the auxiliary lemmas \ref{lem:inversegamma}, ~\ref{lem:variationSmaleparameters} and~\ref{lem:variationNewton}

\begin{proof}[Proof of Lemma~\ref{lem:inversegamma}]
We have that
$
\diff_xf^{-1}\diff_yf=\bbI+\sum_{k=1}^{\infty}\diff_xf^{-1}\frac{\diff_x^{k+1}f(y-x,\ldots,y-x)}{k!}.
$
Now, under the given assumption,
$
\left\|\diff_xf^{-1}\frac{\diff_x^{k+1}f(y-x,\ldots,y-x)}{k!}\right\|\leq \left(\gamma(f,x)\|y-x\|\right)^{k-1}<1
$
for $k\geq 2$, and so, by the the ultrametric inequality, $\|\diff_xf^{-1}\diff_yf-\bbI\|<1$. Therefore 
$
\sum_{k=0}^\infty (\bbI-\diff_xf^{-1}\diff_yf)^k
$
converges, and it does so to the inverse of $\diff_xf^{-1}\diff_yf$. Since, by assumption $\diff_xf$ is invertible, so it is $\diff_yf$. 

Finally, by the invertibility of $\diff_yf$, we have that
$
\diff_yf^{-1}\diff_xf=\sum_{k=0}^\infty (\bbI-\diff_xf^{-1}\diff_yf)^k,
$
and so, by the equality case of the ultrametric inequality, $\|\diff_yf^{-1}\diff_xf\|=1$, as desired.
\end{proof}

\begin{proof}[Proof of Lemma~\ref{lem:variationSmaleparameters}]
We first prove (c) and then (b). (a) follows from (b) and (c) immediately.

(c) We note that under the given assumption, for $k\geq 2$,
\begin{equation}\label{eq:ineqmixedgamma1}
    \left\|\diff_xf^{-1}\frac{\diff_y^kf}{k!}\right\|\leq\gamma(f,x)^{k-1}.
\end{equation}
For this, we expand the Taylor series of $\frac{\diff_y^kf}{k!}$ (with respect $y$) and note that its $l$th term is dominated by
\[
\gamma(f,x)^{k+l-1}\|y-x\|^l,
\]
which, by the ultrametric inequality, gives the above inequality. In this way, for $k\geq 2,$
\[
\left\|\diff_yf^{-1}\frac{\diff_y^kf}{k!}\right\|\leq \left\|\diff_yf^{-1}\diff_xf\right\|\left\|\diff_xf^{-1}\frac{\diff_y^kf}{k!}\right\|\leq \gamma(f,x)^{k-1}
\]
by Lemma~\ref{lem:inversegamma} and~\eqref{eq:ineqmixedgamma1}. Thus $\gamma(f,y)\leq \gamma(f,x)$. Now, due to this, the hypothesis $\gamma(f,y)\|x-y\|<1$ holds, and so, by the same argument, $\gamma(f,x)\leq\gamma(f,y)$, which is the desired equality.

(b) Arguing as in (c), we can show that
\begin{equation}\label{eq:ineqmixedbeta1}
    \left\|\diff_xf^{-1}f(y)\right\|\leq \max\{\|\diff_xf^{-1}f(x)+y-x\|,\gamma(f,x)\|y-x\|^2\}
\end{equation}
by noting that the general term (of the Taylor series of $\diff_xf^{-1}f(y)$ with respect $y$) is dominated by $\gamma(f,x)^{k-1}\|y-x\|^k<\gamma(f,x)\|y-x\|^2$. Now, $\beta(f,y)\leq \left\|\diff_yf^{-1}\diff_xf\right\|\left\|\diff_xf^{-1}f(y)\right\|$, and so, by Lemma~\ref{lem:inversegamma} and~\eqref{eq:ineqmixedbeta1},
\[
    \beta(f,y)
    \leq \max\{\|\diff_xf^{-1}f(x)+y-x\|,\gamma(f,x)\|y-x\|^2\}\leq \max\{\beta(f,x),\|y-x\|\}.
\]
For the equality case, note that, by the same argument, we have
$
\beta(f,x)\leq\max\{\beta(f,y),\|y-x\|\}=\beta(f,y)
$
where the equality on the right-hand side follows from $\beta(f,x)>\|y-x\|$.
\end{proof}
\begin{proof}[Proof of Lemma~\ref{lem:variationNewton}]
(a) follows from combining (b) and (c), and (c) from Lemma~\ref{lem:variationSmaleparameters} (c). We only need to show (b). We use equation~\eqref{eq:ineqmixedbeta1} in the proof of Lemma~\ref{lem:variationSmaleparameters} with $y=\newton_f(x)$. By~\eqref{eq:ineqmixedbeta1} and Lemma~\ref{lem:inversegamma},
\[
\beta(f,\newton_f(x))\leq \max\{\|D_xf^{-1}f(x)+N_f(x)-x\|,\gamma(f,x)\|\newton_f(x)-x\|^2\}.
\]
Now, $\newton_f(x)-x=-\diff_xf^{-1}f(x)$, so the above becomes
$
\beta(f,\newton_f(x))\leq \max\{0,\gamma(f,x)\beta(f,x)^2\},
$
which gives the desired claim.
\end{proof}

\section*{Acknowledgements}

J.G.S. is supported by a Companion Species Research Fellowship funded by the Durlacher Foundation. J.T.-C. is supported by a postdoctoral fellowship of the 2020 ``Interaction'' program of the Fondation Sciences Mathématiques de Paris, and partially supported by ANR JCJC GALOP (ANR-17-CE40-0009).

J.G.S. and J.T.-C. are thankful to Elias Tsigaridas for useful suggestions, and to Evgenia Lagoda for moral support.

\bibliographystyle{plain}

\bibliography{biblio}

\end{document}